\documentclass[reqno,11pt,centertags]{amsart}

\usepackage{amsmath,amsthm,amscd,amssymb,latexsym,upref,enumerate}
\input epsf
\usepackage{epsfig}
\usepackage[T2A,OT1]{fontenc}
\usepackage[ot2enc]{inputenc}
\usepackage[russian,english]{babel}
\usepackage{epic,eepicemu}
\usepackage{tikz}

\newcommand{\bbN}{{\mathbb{N}}}
\newcommand{\bbR}{{\mathbb{R}}}

\newcommand{\bbZ}{{\mathbb{Z}}}
\newcommand{\bbC}{{\mathbb{C}}}

\newcommand{\bbT}{{\mathbb{T}}}

\newcommand{\cA}{{\mathcal{A}}}
\newcommand{\cB}{{\mathcal{B}}}

\newcommand{\cE}{{\mathcal{E}}}
\newcommand{\cF}{{\mathcal{F}}}
\newcommand{\cG}{{\mathcal{G}}}

\newcommand{\cL}{{\mathcal{L}}}

\newcommand{\cM}{{\mathcal{M}}}

\newcommand{\sE}{{\mathsf{E}}}

\newcommand{\fL}{{\mathfrak{L}}}
\newcommand{\fM}{{\mathfrak{M}}}
\newcommand{\fP}{{\mathfrak{P}}}

\newcommand{\fp}{{\mathfrak{p}}}

\newcommand{\e}{{\epsilon}}

\newcommand{\vk}{{\varkappa}}
\newcommand{\vp}{{\varphi}}

\newcommand{\tr}{\text{\rm tr}}
\newcommand{\us}{\underline{s}}

\newcommand{\hus}{\hat{\underline{s}}}

\allowdisplaybreaks \numberwithin{equation}{section}
\newtheorem{theorem}{Theorem}[section]

\newtheorem{lemma}[theorem]{Lemma}

\theoremstyle{definition}
\newtheorem{definition}[theorem]{Definition}

\newtheorem{remark}[theorem]{Remark}

\date{\today}
\title[On point spectrum of Jacobi matrices]{On point spectrum of Jacobi matrices generated by iterations of quadratic polynomials}

\keywords{Jacobi matrix, point spectrum, limit periodic operator, hull, Julia set, Ruelle operator}

\subjclass{47B36, 42C05, 37C30, 37F15}

\author[B.\ Eichinger]{Benjamin Eichinger}

\address{B. Eichinger: School of Mathematical Sciences, Lancaster University, Lancaster LA1 4YF, United Kingdom}

\email{b.eichinger@lancaster.ac.uk}

\thanks{B.\ E.\ was supported by the Austrian Science Fund FWF, project no: P33885}

\author[M.\ Luki\'c]{Milivoje Luki\'c}

\address{M. Luki\'c: Department of Mathematics, Rice University, Houston, TX~77005 and Department of Mathematics, Emory University, Atlanta, GA~30322, USA}

\email{milivoje.lukic@rice.edu}

\thanks{M. L.\ was supported in part by NSF grant DMS--2154563.}

\author[P.\ Yuditskii]{Peter Yuditskii}

\address{P. Yuditskii: Institut f{\"u}r Analysis,
Johannes Kepler Universit{\"a}t, 
Altenberger Strasse 69,
4040 Linz, Austria}

\email{peter.yuditskii@gmail.com}
\thanks{P.\ Y.\ was supported by the Austrian Science Fund FWF, project no: P34414}

\begin{document}

\begin{abstract}
In general, point spectrum of an almost periodic Jacobi matrix can depend on the element of the hull. In this paper, we study the hull of the limit-periodic Jacobi matrix corresponding to the equilibrium measure of the Julia set of the polynomial $z^2-\lambda$ with large enough $\lambda$; this is the leading model in inverse spectral theory of ergodic operators with zero measure spectrum.  We prove that every element of the hull has empty point spectrum.
\end{abstract}

\maketitle

\section{Introduction}

In an ergodic family of operators, ergodicity implies almost sure constancy of the spectrum, as well as a.c., s.c., p.p.\ spectrum \cite{PF}. Historically, the a.c.\ spectrum was first understood in the setting of periodic operators, and the pure point spectrum gained prominence in the context of Anderson localization. Singular continuous spectrum was originally perceived as an anomaly, but then discovered to be generic in many settings \cite{Simon,AvilaDamanik,BoshernitzanDamanik}. The study of spectral type of ergodic operators is now a vast area of research; it allows only a few general results, such as Kotani theory \cite{Kotani,Remling}, and a lot of research is centered on prominent models such as the Anderson model \cite{Anderson}, the Almost Mathieu operator \cite{JM}, and the Fibonacci Hamiltonian \cite{Damanik}; the dependence of spectral type on the ergodic space and the size and arithmetic properties of the parameters are of interest, and the analyses are highly model-dependent.

Let us focus on the setting of almost periodic Jacobi matrices, recalling that a Jacobi matrix
\[
(Ju)_n = a_{n} u_{n-1} + b_n u_n + a_{n+1} u_{n+1}
\]
is almost periodic if the set of its translates $S^n J S^{-n}$ is precompact in the set of bounded operators on $\ell^2(\bbZ)$; here $S$ denotes the shift on $\ell^2(\bbZ)$. The closure of the set of translates is called the hull of $J$. Equivalently, $J$ is almost periodic if there exist continuous maps $\cA, \cB : G \to \bbR$ on a compact Abelian group $G$ with $\cA \ge 0$ and $\omega, \theta \in G$ such that
\[
a_n = \cA(\theta + n \omega), \qquad b_n = \cB(\theta + n\omega).
\]
By passing to a subgroup, one reduces to the case when the orbit $\{\theta + n\omega \mid n\in\bbZ\}$ is dense in $G$ and obtains an ergodic family, with the hull in bijective correspondence to the phase $\theta \in G$. For almost periodic operators, some of the "almost sure" statements extend to "sure" statements:  the spectrum is constant, and by a result of Last--Simon \cite{LastSimon}, the absolutely continuous spectrum is constant. However, other spectral types are more sensitive to perturbations, and can be altered by changing the phase $\theta$.

For instance, the Almost Mathieu operator
\[
(J u)_n = u_{n-1} + 2 \lambda \cos(\theta + n\omega) u_n + u_{n+1},
\]
with irrational $\omega \in \bbT$ is an almost periodic operator, with hull indexed by $\theta \in \bbT$. In the subcritical case $\lambda <1$, it has a.c.\ spectrum, but in the supercritical case $\lambda > 1$, its spectrum is s.c.\ or p.p.\ depending on the phase  \cite{JitomirskayaSimon,Jitomirskaya99}. A remarkable result of Jitomirskaya \cite{Ji} proves that in the critical case $\lambda=1$, Almost Mathieu operator has empty p.p.\ spectrum for every choice of $\theta$, after partial results \cite{Avila}.

Jacobi matrices generated by iterations of quadratic polynomials are the leading model in inverse spectral theory. They have attracted much attention, providing almost periodic operators with zero Lebesgue measure spectrum; cf. \cite{BellBeMou82,BarnGerHarr82, BaGerHarr85, BelGerVolYu05, PeVolYu06}. In this paper, we prove absence of pure point spectrum in the hull for this model and view this as a counterpart of Jitomirskaya's theorem. 

To describe the model, start with an even quadratic polynomial
\[
T(z) = z^2 - \lambda, \qquad \lambda  > 2.
\]
Its Julia set
\[
\sE = \bbC \setminus \{ z \in \bbC \mid T^{\circ n}(z) \to \infty, n \to\infty\}
\]
is a subset of $\bbR$ of zero one-dimensional Lebesgue measure; we denote $T^{\circ n} = T \circ T^{\circ (n-1)}$.  Denote by $\mu_\sE$ the equilibrium measure for the set $\sE$. The measure $\mu_\sE$ generates orthogonal polynomials which obey the Jacobi recurrence relation
\[
x p_n(x) = a_n p_{n-1}(x) + a_{n+1} p_{n+1}(x)
\]
(note that $b_n = 0$ by symmetry), with $a_0 = 0$. For $\lambda > 3$, the coefficients $a_n$ form a limit periodic sequence:  for an arbitrary $\epsilon>0$ there exist $n_0$ such that 
$$
|a_{k+s2^n}-a_k|\le \epsilon, \quad \forall n\ge n_0
$$
uniformly in $k$ and $s$. Due to this property the sequence can be extended to a continuous function on the space of dyadic integers $\bbZ_2$
\begin{equation}\label{eqnZ2continuity}
a_\vk=\lim_{k_j\to \vk} a_{k_j},\quad k_j\to\vk\in \bbZ_2.
\end{equation}
(recall that $\bbZ_2$ is the completion of $\bbZ$ with respect to the metric $d(m,n) = 1/2^j$ if $m-n \in 2^j \bbZ \setminus 2^{j+1} \bbZ$). This extends the half-line Jacobi matrix associated to $\mu_\sE$ to a full-line limit-periodic Jacobi matrix $J_0$, and describes the hull of this matrix as $\{ J_\vk \mid \vk \in \bbZ_2 \}$ where
$$
J_\vk e_n=a_{\vk+n}e_{n-1}+a_{\vk+n+1}e_{n+1},
$$
and $\{e_n\}_{n\in\bbZ}$ is the standard basis in $\ell^2(\bbZ)$.

All $J_\varkappa$ have spectrum $\sE$, so they have purely singular spectrum. It was proved \cite{SY} that $J_0$ has no eigenvalues; this proof used the observation $a_0 = 0$ to decompose $J_0$ as a direct sum of half-line Jacobi matrices, and used Ruelle theory to study the other half-line matrix.  Since $a_\vk > 0$ for $\vk \neq 0$, $J_\vk$ has no such decomposition if $\vk \notin \bbZ$. Nonetheless, we prove the following:
\begin{theorem}\label{mt}
For $\lambda > 3$ and $\varkappa \in \bbZ_2$, the operator $J_\varkappa$ on $\ell^2(\bbZ)$ has no eigenvalues. It has purely singular continuous spectrum.
\end{theorem}

The key tool are matrix Ruelle operators, see Definition \ref{def:1}. A matrix generalization of Ruelle theory \cite{Bo} seems to be a hard problem, due to noncommutativity. For instance, whereas the multiplicative representation of Schur functions, after separation of a Blaschke product, is a simple observation, its matrix valued analog is the multiplicative theory of $J$-contractive matrices \cite{ArovDym}, which includes the de Branges theory of canonical systems \cite{deBranges}. However, we were able to reduce our problem to matrix Ruelle operators with special rational weights. This approach motivates a more systematic study of matrix Ruelle operators.

Dyadic integers have series expansions $\vk = \sum_{k=0}^\infty \e_k 2^k$, $\e_k \in \{0,1\}$, and in our analysis, it is necessary to separate the hull into elements for which the dyadic expansion contains stretches of zeros and ones of uniformly bounded length, $\cF_\sigma$, and its complement, $\cG_\delta$. Different arguments are used in the two cases.

It is an open problem whether the spectral measure corresponding to an element $\varkappa$ with periodic digits, such as $\varkappa =1+ \sum_{k=0}^\infty 2^{2k+1}$ (in $\bbZ_2$), is mutually singular to the spectral measure of $J_0$. The operator $J_\varkappa$ corresponds to a matrix Ruelle operator, and using Ruelle theory for analogy suggests that different weights should correspond to mutually singular measures. If so, this would be another example of operators in the same hull which are not unitarily equivalent, in strong contrast to the results of \cite{SY2} for reflectionless operators on regular Widom spectra with DCT property, where all operators in the same hull are unitarily equivalent to each other.

Besides the intrinsic interest in the spectral type, absence of eigenvalues is used in an upcoming preprint \cite{ELY} in the analysis of the scaling behavior of Christoffel--Darboux kernels corresponding to $\mu_\sE$.

\noindent \textbf{Structure of the paper:}  In Section 2 we introduce matrix Ruelle operators associated to the hull $\{J_\vk\}$. We divide the hull into elements for which $\vk$ contains stretches of zeros and ones of uniformly bounded length, $\cF_\sigma$, and its complement, $\cG_\delta$. In Section 3 we prove Theorem \ref{mt} in the generic case, i.e., $\vk\in \cG_\delta$. The case $\vk\in \cF_\sigma$ requires uniform boundedness from below of a certain sequence of \textit{matrix valued} test functions, which is proved in Section 4. With these bounds we conclude the proof of Theorem \ref{mt} in Section 4.

\noindent \textbf{Acknowledgment:} We thank David Damanik for helpful information. P.Y.  thanks  Paul F.X. M\"uller for useful discussions.


\section{Hull $\{J_\vk\}$ and matrix Ruelle operators}

\subsection{Background: Ruelle operators and half-line Jacobi matrices} 
To a H\"older function $\varphi$ one associates the Ruelle operator
\[
(\cL_\vp f)(x)=\sum_{T(y)=x}e^{\vp(y)}f(y), \quad f\in C(\sE).
\]
Let $\rho=\rho(\varphi)$ be the spectral radius of $\cL_\vp$. According to the Ruelle version of the Perron-Frobenius theory, $\rho$ is an eigenvalue for the operator and its adjoint
\[
\frac 1{\rho} \cL_\vp h=h,\quad \frac 1{\rho} \cL^*_\vp \nu=\nu,
\]
with a positive continuous function $h\in C(\sE)$ and a (positive) measure $\nu$. With a slight abuse of notation, we will denote by $\cL_j$, $j=0,1,2$ the Ruelle operators
\[
(\cL_jf)(x) = \sum_{T(y)=x} \frac{f(y)}{T'(y)^j}.
\]
In particular, the equilibrium measure $\mu_\sE$ for the set $\sE$ is an eigenvector for $\cL_0^*$,
\begin{equation}\label{eq:1}
\int f(x) \mu_\sE(dx)=\frac 1 2\int (\cL_0f)(x) \mu_\sE(dx).
\end{equation}
We denote by $J_+$ the half-line Jacobi matrix corresponding to $\mu_\sE$.

The extension of $J_+$ to a full-line operator $J_0$ uses the continuity relation \eqref{eqnZ2continuity}. In particular, the extension of the sequence $(a_k)_{k \in \bbZ_+}$ to $(a_k)_{k\in \bbZ}$ is given by
\[
a_{-k}=\lim_{n\to\infty}a_{2^n-k}.
\]
Since $a_0=0$, this gives a direct sum decomposition $J_0= J_+ \oplus J_-$.
 We use $\bbZ_+ = \{n\in\bbZ \mid n\ge 0\}$ and $\bbZ_- = \bbZ \setminus \bbZ_+$. Note that all diagonal Jacobi coefficients are zero.

It was proved in \cite{SY} that the canonical spectral measure $\nu(dx)$ of $J_-$ is the eigenmeasure of the adjoint of the Ruelle operator
\[
\cL_2 f(x)=\sum_{T(y)=x}\frac{f(y)}{T'(y)^2},
\]
i.e.,
$\cL_2^*\nu=\rho_2 \nu$ and $\rho_2=\frac 1{2 a_{-1}^2}$. The corresponding orthogonal polynomials $p_n(x,\nu)$ possess the renormalization relation
\begin{equation}\label{eqnJminusrenorm}
p_{2n+1}(x,\nu)=\frac{x}{a_{-1}}p_n(T(x),\nu).
\end{equation}

\subsection{Shift in $\bbZ_2$ and matrix Ruelle operators associated to $J_\vk$}

We will use the operator $V$ on $\ell^2(\bbZ)$ given by $Ve_n = e_{2n}$, and denote by $V_\pm$ its  restrictions to $\ell^2(\bbZ_\pm)$.

\begin{lemma} 
For all $z \in \bbC \setminus \sE$,
\begin{equation}\label{eq:28dec-1}
V^*(J_{2\vk}-z)^{-1}V=\frac 1 2 T'(z)(J_\vk-T(z))^{-1}
\end{equation}
\end{lemma}
\begin{proof}
Due to \eqref{eq:1} 
the operator $f(x)\mapsto f(T(x))$ is an isometry in $L^2_{\mu_{\sE}}$. Its adjoint is $\frac 1 2 \cL_0$ viewed as an operator in $L^2_{\mu_{\sE}}$. With the standard  eigenfunction expansion $\ell^2(\bbZ_+) \to L^2_{\mu_\sE}$,  $e_n \mapsto p_n(\cdot,\mu)$, the operator $f\mapsto f\circ T$ is conjugated to $V_+$. Thus the identity in the spectral representation 
\begin{equation}\label{eqn:5aug}
\frac 1 2\sum_{T(y)=x}\frac{f(T(y))}{y-z}=\frac{1}{2}\frac{T'(z)}{x-T(z)}f(x)
\end{equation}
can be rewritten as
$$
V_+^*(J_+ - z)^{-1}V_+ = \frac 1 2 T'(z)(J_+ - T(z))^{-1}.
$$
Likewise, the standard eigenfunction expansion for $J_-$ can be written as $\ell^2(\bbZ_-) \to L^2_\nu$, $e_{-1-n} \mapsto p_n(\cdot,\nu)$, and due to \eqref{eqnJminusrenorm}, the adjoint of $V_-$ in the spectral representation takes the form
\[
a_{-1}\cL_1 f(x)= a_{-1}\sum_{T(y)=x}\frac{f(y)}{T'(y)},
\]
used on continuous functions $f$ and extended by continuity to $L^2_{\nu}$. 
Similarly to $J_+$, the identity
\[
a_{-1}\sum_{T(y)=x}\frac 1{y-z}\frac{yf(T(y))}{a_{-1}T'(y)}=
\frac 1 2\sum_{T(y)=x}\frac 1{y-z}f(x)=\frac 1 2\frac{T'(z)}{x-T(z)}f(x)
\]
yields the identity
\[
V_-^*(J_- - z)^{-1}V_- = \frac 1 2 T'(z)(J_- - T(z))^{-1}.
\]
Taking direct sums yields
\[
V^*(J_0 - z)^{-1}V=\frac 1 2 T'(z)(J_0  - T(z))^{-1}.
\]
Using the shift operator $S e_n = e_{n+1}$ and noting $VS=S^2V$, we have
\[
V^*(S^{-2k}J_0S^{2k}-z)^{-1}V=\frac 1 2 T'(z)(S^{-k}J_0S^{k}-T(z))^{-1}.
\]
Therefore for a sequence $k_s\to\vk\in \bbZ_2$ we have $S^{-k_s}J_0S^{k_s}\to J_\vk$ and consequently \eqref{eq:28dec-1}.
\end{proof}

Let $\cE:\bbC^2\to \ell^2(\bbZ)$ be defined by
$$
\cE\begin{pmatrix}C_{-1}\\C_0\end{pmatrix}=\begin{pmatrix}e_{-1}&e_0 \end{pmatrix}
\begin{pmatrix}C_{-1}\\C_0\end{pmatrix}
=C_{-1}e_{-1}+C_0e_0.
$$
The spectral matrix-measure  $\Sigma_\vk(dx)$ associated to $J_\vk$ is defined by
$$
\cE^*(J_\vk-z)^{-1}\cE=\int\frac{\Sigma_\vk(dx)}{x-z}.
$$
We introduce the shift operation $\us: \bbZ_2\to \bbZ_2$ as
$$
\us\vk=\sum_{k=0}^\infty\e_{k+1}2^k, \quad \vk=\sum_{k=0}^\infty\e_{k}2^k
$$
We also introduce a modified shift on $\bbZ_2$ defined by
$$
\hus\vk=\begin{cases}\us\vk, &\vk=2\us\vk\\
1+\us\vk,& \vk=1+2\us\vk
\end{cases}
$$
which naturally appears in the following relation:

\begin{lemma} \label{l:26n4-1}
Define
$$
\fp_\vk(x)=\begin{pmatrix}
 a_{\vk-1}&0\\
a_{\vk}&1 \end{pmatrix}
\begin{pmatrix}
2/{T'(x)}&0\\
0&1 \end{pmatrix}, \quad \vk=2\us\vk,
$$
and
$$
\fp_\vk(x)=\begin{pmatrix}
1&a_{\vk}\\
0& a_{\vk+1} \end{pmatrix}
\begin{pmatrix}
1&0\\
0&2/{T'(x)} \end{pmatrix}, \quad \vk=1+2\us\vk.
$$
The following renormalization equation holds 
\begin{equation}\label{eq:29dec-2}
\Sigma_{\vk}(dx)=\frac 1 2 \fp_{\vk}^*(x)\cL_0^*\Sigma_{\hat\us\vk}(dx)\fp_{\vk}(x).
\end{equation}
\end{lemma}

\begin{proof}
Let $\vk=2\us\vk$, i.e., $\e_0=0$.  
We will pass to spectral matrix-measures. Recall that in the spectral representation for two-sided Jacobi matrices, for $k\ge 0$, we have
$$
e_k\mapsto \begin{pmatrix}
-a_0q^{+}_{k}\\
p^{+}_{k}
\end{pmatrix},\quad e_{-1-k}\mapsto \begin{pmatrix}
p^{-}_{k}\\
-a_0q^{-}_{k}
\end{pmatrix},
$$
where $p_k^\pm, q_k^\pm$ are orthogonal polynomials of the first and second kind generated by the half-line restrictions $(J_\vk)_\pm$ of $J_\vk$ to $\ell^2(\bbZ_\pm)$. 
In particular
\begin{align}\label{eq:Ben1}
	e_{-2}\mapsto
	\begin{pmatrix}
		\frac{x}{ a_{\vk-1}}\\
		-\frac{a_{\vk}}{a_{\vk-1}}
	\end{pmatrix}=
	\begin{pmatrix}
		\frac{T'(x)}{2 a_{\vk-1}}\\
		-\frac{a_{\vk}}{a_{\vk-1}}
	\end{pmatrix},\quad e_0\mapsto
	\begin{pmatrix}
		0\\1
	\end{pmatrix}.
\end{align}
Therefore \eqref{eq:28dec-1} implies
$$
\begin{pmatrix}
\frac{T'(x)}{2 a_{\vk-1}}&0\\
-\frac{a_{\vk}}{a_{\vk-1}}&1 \end{pmatrix}^*
\Sigma_\vk(dx)
\begin{pmatrix}
\frac{T'(x)}{2 a_{\vk-1}}&0\\
-\frac{a_{\vk}}{a_{\vk-1}}&1 \end{pmatrix}=\frac 1 2 \cL_0^*\Sigma_{\us\vk}(dx).
$$
We get \eqref{eq:29dec-2} in the given case.

Let $\vk=1+2\us\vk$, i.e., $\e_0=1$.
Now we have
\begin{equation}\label{eq:13jan-1}
V^*S(J_{\vk}-z)^{-1}S^{-1}V=
V^*(J_{\vk-1}-z)^{-1}V=
\frac 1 {2} T'(z)(J_{\us\vk}-T(z))^{-1}.
\end{equation}
We define an involution $\tau e_n=e_{-1-n}$, $\tau^2=I, \tau^*=\tau$. Note that
$$
\tau V\tau e_n=\tau Ve_{-1-n}=\tau e_{-2-2n}=e_{2n+1}=SV e_n
$$
and
$$
\tau S\tau e_n=\tau Se_{-1-n}=\tau e_{-n}=e_{-1+n}=S^{-1}e_n.
$$
Denote $J^\sharp=\tau J\tau$. Conjugating \eqref{eq:13jan-1} with $\tau$ we get
\begin{equation*}\label{}
V^*S^{-2}(J^\sharp_{\vk}-z)^{-1}S^2V=S^{-1}V^*(J^\sharp_{\vk}-z)^{-1}VS=\frac 1 {2} T'(z)(J^{\sharp}_{\us\vk}-T(z))^{-1}
\end{equation*}
and therefore \eqref{eq:13jan-1} gets the form
\begin{align}\label{eq:13jan-2}
V^*(J^\sharp_{\vk}-z)^{-1}V
=&\frac 1 {2} T'(z)(SJ^{\sharp}_{\us\vk}S^{-1}-T(z))^{-1} \nonumber\\
=&\frac 1 {2} T'(z)(\tau S^{-1}J_{\us\vk}S\tau-T(z))^{-1}\nonumber\\
=&\frac 1 {2} T'(z)(J_{\us\vk+1}^\sharp-T(z))^{-1}.
\end{align}
Using 
$$
\tau\cE=\cE\sigma_1, \quad\sigma_1=\begin{pmatrix}
0&1\\1&0
\end{pmatrix},
$$
we have
\begin{equation}\label{eq:13jan-3}
\Sigma^{\sharp}(dx)=\sigma_1\Sigma(dx)\sigma_1,
\quad\text{where}\quad
\int\frac{\Sigma^{\sharp}(dx)}{x-z}=\cE^*(J^\sharp-z)^{-1}\cE.
\end{equation}
Thus \eqref{eq:13jan-2} yields the following relation for spectral matrix-functions
$$
\begin{pmatrix}
\frac{T'(x)}{2 a_{\vk+1}}&0\\
-\frac{a_{\vk}}{a_{\vk+1}}&1 \end{pmatrix}^*
\Sigma_\vk^{\sharp}(dx)\begin{pmatrix}
\frac{T'(x)}{2 a_{\vk+1}}&0\\
-\frac{a_{\vk}}{a_{\vk+1}}&1 \end{pmatrix}=\frac 1 2 \cL_0^*\Sigma_{\us\vk+1}^\sharp(dx)
$$
Using \eqref{eq:13jan-3}, we obtain \eqref{eq:29dec-2} in the second case.
\end{proof}

\begin{remark}\label{r:25may-1}
The modified shift obeys
\[
\vk =\begin{cases} 2 \hus \vk & \vk\text{ even} \\
-1 + 2 \hus \vk & \vk\text{ odd}
\end{cases}
\]
so this is equivalent to the standard shift applied to $-\vk$. In other words
\begin{equation}\label{eq:21may25}
\hus\vk = - \us (-\vk).
\end{equation}
\end{remark}

\begin{definition}
	We denote
	$$
	\hus^n\vk=\sum_{k=0}^\infty\e^n_k 2^k.
	$$
	Define
	$$
	\kappa=\kappa(\vk)=\sum_{n=0}^\infty\e^n_0 2^n. 
	$$
\end{definition}
It gives the identity related to the standard shift
$$
\kappa(\hus\vk)=\us\kappa(\vk).
$$

We decompose $\bbZ_2$ into two shift invariant subsets. 
\begin{definition}
For an arbitrary $N\in\bbZ_+$ we define $\cF_N$ as a collection of $\vk\in \bbZ_2$ such that $\vk$  has at most $N$ consecutive zeros or ones in its dyadic representation.
\end{definition}

\begin{lemma}\label{lem:26}
	Let $\vk\in \cF_{N}$. Then $\kappa(\vk)\in \cF_{N+1}$.
\end{lemma}

\begin{proof} We will prove an explicit relation \eqref{eq25m-1} between components of $\vk$ and $\kappa(\vk)$.
	Recall that
	$$
	\hus^n \vk=-\us^n(-\vk).
	$$
	Further, let $\vk=\sum_{k\ge 0}\e_k 2^k$. Since
	$$
	\sum_{k\ge 0}\e_k 2^k+\sum_{k\ge 0}(1-\e_k)2^k=-1
	$$
	we have
	$$
	-\vk=1 + \sum_{k\ge 0}(1-\e_k) 2^k.
	$$
	
	Let $\vk\in\cF_N$. Let $n_0 = \min\{ k \in \bbZ_+ \mid \e_k = 1\}$. Note $n_0 \le N$. By repeatedly applying the modified shift, we see that $\e_0^n = 0$ for all $n < n_0$ and $\e_0^{n_0}=1$. Applying $n_0$ shifts we get an odd number
	$$
	\hus^{n_0}\vk=1+\sum_{k\ge 1} \e_{k+n_0} 2^k,\quad n_0\le N.
	$$
	Then
	$$
	-\hus^{n_0}\vk= 1 + \sum_{k\ge 1}(1-\e_{k+n_0}) 2^k
	$$
	and
	$$
	\us^n(-\hus^{n_0}\vk)=\sum_{k\ge 0}(1-\e_{k+n_0+n}) 2^k,\quad n\ge 1.
	$$
	We obtain
	$$
	\hus^{n+n_0}\vk = 1 + \sum_{k\ge 0}\e_{k+n+n_0} 2^k,
	$$
from which we read off all the remaining digits of $\kappa(\vk)$ to be
\begin{equation}\label{eq25m-1}
\e^{n+n_0}_0=\e_{n+n_0}+1\mod 2,\quad n\ge 1.
\end{equation}
	
	Note that $\e^{n_0}_0=1$ and $\e^n_0=0$ for $n<n_0$. So, if $\e_{k+n_0}=0$ for $1\le k\le N$ we get $N+1$ consecutive ones in the decomposition of $\kappa(\vk)$.
\end{proof}

Evidently, each $\cF_N$ is a perfect closed set. We denote
$$
\cF_\sigma=\cup_{N\ge 1}\cF_N,\quad \cG_\delta=\bbZ_2\setminus \cF_\sigma.
$$

We will use duality with respect to the pairing
$$
\int \tr \{g(x)\Sigma(dx)\},
$$
where $g$ is $2\times 2$ continuous matrix function on $\sE$. Motivated by this duality and Lemma~\ref{l:26n4-1},  we define matrix Ruelle operators.

\begin{definition}\label{def:1}
 We introduce the  following Ruelle operator with matrix weight, acting on $2\times 2$ continuous  matrix-functions
\begin{equation}\label{eq:19jan-3}
(\cM_\vk g)(x)
=\frac 1 2\sum_{T(y)=x}
\fp_{\vk}(y) g(y)\fp^*_{\vk}(y).
\end{equation}
\end{definition}
Note that 
$$
\int \tr \{g(x)(\cM_\vk^*\Sigma)(dx)\}=\int \tr\{ (\cM_\vk g)(x)\Sigma(dx)\},
$$
and $g(x)\ge 0$ implies $(\cM_\vk g)(x)\ge 0$. In particular \eqref{eq:29dec-2} has the form
\begin{equation}\label{eq:29dec-2cor}
\Sigma_{\vk}(dx)=(\cM_{\vk}^*\Sigma_{\hat\us\vk})(dx).
\end{equation}

\section{Main theorem: generic case}

Recall that 
 $$
 \sE\subset [-\xi,\xi],\quad\xi=\frac{1+\sqrt{1+4\lambda}} 2, 
  $$
 and $\pm \xi \in \sE$. In fact, $T^{-1}([-\xi,\xi]) \subset [-\xi,\xi]$ and $\sE = \cap_{n=0}^\infty (T^{-1})^{\circ n}([-\xi,\xi])$.  Moreover, note that $T(0)=-\lambda < - \xi$.

First we show that the collection of rational functions with simple poles at $T^{\circ k}(0)\in \bbR\setminus [-\xi,\xi]$ is invariant for the action of 
$\cM_\vk$ (we do not consider  its closure). 

\begin{lemma}\label{l:13feb-1}
Let $A$ be a constant $2\times 2$ matrix and
$$
g_z(x)=\frac{A}{x-z}, \quad z\in \bbC\setminus\sE.
$$
Then
\begin{equation}\label{eq:19jan-4}
(\cM_\vk g_{T^{\circ k}(0)})(x)=\frac{\Phi_\vk A\Phi_{\vk}^*}{-T^{\circ k}(0)(x-T(0))}
+\frac{T^{\circ k}(0)\fp_\vk(T^{\circ k}(0)) A\fp_\vk(T^{\circ k}(0))^*}{x-T^{\circ (k+1)}(0)}.
\end{equation}
where
\begin{equation}\label{eq:25feb-7}
\Phi_\vk=\left.(z\fp_\vk(z))\right|_{z=0}
=\begin{cases}\begin{pmatrix}
 a_{\vk-1}&0\\
a_{\vk} &0\end{pmatrix}, \quad \vk=2\us\vk,\\
\\
\begin{pmatrix}
0&a_{\vk}\\
0& a_{\vk+1} \end{pmatrix}, \quad \vk=1+2\us\vk.
\end{cases}
\end{equation}
\end{lemma}

\begin{proof}
Let us apply $\cM_\vk$ to $g_z(x)$.
To this end we consider an identity between rational functions in $z$,
\begin{equation}\label{eq:18jun-1}
\frac 1 2
\fp_{\vk}(z) \frac{A T'(z)}{x-T(z)}\fp^*_{\vk}(z)=(\cM_\vk g_z)(x)+\frac{B_0(x)}{z},
\end{equation}
where
$$
B_0(x)=\frac{\Phi_\vk A\Phi^*_\vk}{x-T(0)},
$$
verified by comparing principal parts, since both sides decay at $\infty$. Set $z=T^{\circ k}(0)$ and we obtain \eqref{eq:19jan-4}. 
\end{proof}

Define  the function $w^n_x(z)$ in $z$, $n\ge 0$,  by 
\begin{equation}\label{eq:6aug-2}
w^n_x(z):=\frac 1{2^n}\frac{(T^{\circ n})'(z)}{x-T^{\circ n}(z)}=\frac{z T(z)\cdots T^{\circ(n-1)}(z)}{{x-T^{\circ n}(z)}},\quad x\in[-\xi,\xi],
\end{equation}
where by definition $T^{\circ 0}(z)=z$.  Note that $w_x^n$ has the partial fraction decomposition
\begin{equation}\label{eqn:wxnpartial}
w_x^n(z) = \frac 1{2^n} \sum_{T^n(y) = x} \frac 1{y-z},
\end{equation}
in particular $w_x^n$ is a Herglotz function.

\begin{remark} 
We can rewrite \eqref{eq:19jan-4} in terms of $w^k_x(T(0))$ as follows
\begin{align}\nonumber\label{eq:18mar-1}
\cM_\vk w_x^{k-1}(T(0)) A=&w_{0}^{k-1}(T(0))w_x^0(T(0))\Phi_\vk A\Phi_\vk^*\\+&w_x^k(T(0))
\fp_\vk(T^{\circ k}(0))A
\fp_\vk(T^{\circ k}(0))^*.
\end{align}
\end{remark}


\begin{lemma}\label{h:24-1} 
Let 
$
 \fM_{n,\vk}=\cM_{\hus^{n-1}\vk}\cdots \cM_{\vk},
$
so that 
$$
(\fM_{n,\vk}^*\Sigma_{\hat \us^n\vk})(dx)=\Sigma_{\vk}(dx).
$$
Define
$$
h_n(x)= (\fM_{n,\vk} h_0)(x), \quad h_0(x)=w_x^0(T(0)) I. 
$$
Then
the sequence $h_n(x)$  is uniformly bounded on $[-\xi,\xi]$. 
\end{lemma}

\begin{proof}
Acting one by one by the operators $\cM_{\hus^k\vk}$ and applying \eqref{eq:18mar-1} we obtain that $h_n(x)$ is of the form
\begin{align}\label{eq:Ben2}
	h_n(x)=\sum_{k=0}^n w^k_x(T(0)) A_k^{n},
\end{align}
where $A_k^n$ are non-negative matrix coefficients.

Further, on $[-\xi,\xi]$ the function
\begin{equation*}
w^0_x(T(0))=\frac{1}{x-T(0)}=\frac{1}{x+\lambda}
\end{equation*}
 is two-sided bounded 
\begin{equation*}
\frac{1}{\lambda+\xi}\le w^0_x(T(0))\le\frac{1}{\lambda-\xi}.
\end{equation*}
From \eqref{eqn:wxnpartial} we obtain
$$
w^k_x(T(0))
=\frac 1{2^k}\sum_{T^{\circ k}(y)=x}\frac{1}{y+\lambda}
$$
and since $\{y:T^{\circ k}(y)=x\}\subset[-\xi,\xi]$, we have
\begin{equation}\label{eq:27mar-1}
\frac{1}{\lambda+\xi}\le w^k_x(T(0))\le\frac{1}{\lambda-\xi}
\end{equation}
for all $k\ge 0$.

Using the upper bound in \eqref{eq:27mar-1} we have
\begin{equation}\label{eq:27mar-2}
\tr\, h_n(x)\le\frac{1}{\lambda-\xi}\sum_{k=0}^n \tr A_k^{n}.
\end{equation}
On the other hand by definition
\begin{align*}
\int\tr (h_n(x)\Sigma_{\hat \us^n\vk}(dx))=\int\tr  ((\fM_{n,\vk} h_0)(x)\Sigma_{\hat\us^n\vk}(dx))=\int \tr(h_0(x)\Sigma_\vk(dx))
\end{align*}
and therefore
\begin{align*}
\int \tr(h_0(x)\Sigma_\vk(dx))=\int \tr (h_n(x)\Sigma_{\hat\us^n\vk}(dx))
\ge \frac{1}{\lambda+\xi} \sum_{k=0}^n\tr(A^n_k\int \Sigma_{\hat\us^n\vk}(dx))
\end{align*}
Since all matrix measures are normalized
$$
\int \Sigma_{\hat\us^n\vk}(dx)=I,
$$
we obtain
\begin{equation}\label{eq:27mar-3}
\sum_{k=0}^n \tr A_k^{n} \le (\lambda+\xi)\int \tr(h_0(x)\Sigma_\vk(dx)).
\end{equation}
Combining \eqref{eq:27mar-2} and \eqref{eq:27mar-3} we get
\begin{equation}\label{eq:27mar-4}
\tr\, h_n(x)\le \frac{\lambda+\xi}{\lambda-\xi}\int \tr(h_0(x)\Sigma_\vk(dx)).
\end{equation}
\end{proof}

\begin{remark}
To have estimation for $\tr\,h_n(x)$ from below we write
\begin{equation}\label{eq:27mar-2b}
\tr\, h_n(x)\ge\frac{1}{\lambda+\xi}\sum_{k=0}^n \tr A_k^{n}.
\end{equation}
Also
\begin{align*}
\int \tr(h_0(x)\Sigma_\vk(dx))=\int \tr (h_n(x)\Sigma_{\hat\us^n\vk}(dx))
\le \frac{1}{\lambda-\xi} \sum_{k=0}^n\tr A^n_k
\end{align*}
In combination with \eqref{eq:27mar-2b}
\begin{equation}\label{eq:27mar-4b}
\tr\, h_n(x)\ge \frac{\lambda-\xi}{\lambda+\xi}\int \tr(h_0(x)\Sigma_\vk(dx)).
\end{equation}
However it does not mean that the matrices  $h_n(x)$ themselves are uniformly bounded from below (i.e., $h_n(x)\ge \delta\cdot I$, $\delta>0$). For special $\vk\in \cF_N$ such boundedness is shown in
Lemma \ref{l:21may-1}.
\end{remark}

\begin{theorem}\label{thm:Gdelta}
If $\vk\in\cG_\delta$, then the spectrum of $J_\vk$ is singular continuous.
\end{theorem}
\begin{proof}
Assume  that for some  $x_0\in \sE$,
\begin{equation}\label{eq:26feb-1}
\tr\Sigma_{\vk}(\{x_0\})>0.
\end{equation}
Since $\vk\in \cG_\delta$, we also have that $-\vk\in \cG_\delta$, so the dyadic expansion of $-\vk$ contains arbitrarily long series of consecutive zeros or ones. Using \eqref{eq:21may25}, in the first case there is a subsequence $I_\vk \subset \bbZ_+$ such that
$$
\lim_{n\in I_\vk}\hat\us^n\vk=0
$$
(we use this notation for the limit along the subsequence $I_\vk$). Alternatively, there is $I_\vk$ such that
$$
\lim_{n\in I_\vk}\hat\us^n\vk=1.
$$

Note that for a test function $g(x)$
\begin{align}\label{eq:14mar-1}\nonumber
\int g(x)\tr \left ((h_n(x) \Sigma_{\hat\us^n\vk}(dx)\right)
\nonumber=&
\int g(x)\tr \left ((\fM_{\hat\us^n\vk}  h_0)(x) \Sigma_{\hat\us^n\vk}(dx)\right)\\
\nonumber=&
\int \tr  \left((\fM_{\hat\us^n\vk}((g\circ T^{\circ n})h_0))(x) \Sigma_{\hat\us^n\vk}(dx)\right)\\
=&\nonumber
\int \tr\left( g(T^{\circ n}(x)) h_0(x)  
(\fM_{\hat\us^n\vk}^*\Sigma_{\hat\us^n\vk})(dx)\right)
\\=&
\int
g(T^{\circ n}(x))\tr\left( h_0(x)  
\Sigma_\vk(dx)\right)
\end{align}
We choose a subsequence $I_{\vk,x_0}\subset I_\vk$ such that
$$
\lim_{\substack{n\in I_{\vk,x_0}}} 
T^{\circ n}(x_0)=y_0.
$$
Now, let us fix an arbitrary small $\delta$-vicinity of $y_0$, $V_\delta= ({y_0-\delta},{y_0+\delta})$. 
For sufficiently large $n\in I_{\vk,x_0}$ we have
$T^{\circ n}(x_0)\in V_\delta$ and therefore by \eqref{eq:14mar-1}
\begin{equation*}\label{eq:27mar-0}
\int_{V_\delta} \tr (h_n(x)\Sigma_{\hat\us^n\vk}(dx))
=\int_{(T^{\circ n})^{-1}V_\delta} \tr (h_0(x)\Sigma_{\vk}(dx))
 \ge \tr (h_0(x_0)\Sigma_{\vk}(\{x_0\}))
\end{equation*}
In a combination with \eqref{eq:27mar-4} 
we have
$$
C\tr \Sigma_{\hat\us^n\vk}(V_\delta)
 \ge \tr (h_0(x_0)\Sigma_{\vk}(\{x_0\}))
$$
with a uniform constant $C$.
Passing to the limit in $n$, by continuity in $\bbZ_2$, we get for some value of $\e \in \{0,1\}$,
$$
C\tr \Sigma_{\e}(V_\delta)
 \ge \tr (h_0(x_0)\Sigma_{\vk}(\{x_0\})).
$$
Since $\delta$ is arbitrarily small, the spectral matrix measure $\Sigma_{\e}$ associated to the Jacobi matrix  $J_\e$ has a non-trivial mass at $y_0$. 
Since the spectrum of both matrices $J_0$ and $J_{1}$ is singular continuous,
we arrive to a contradiction. Thus, the initial assumption \eqref{eq:26feb-1} is false, that is, $\Sigma_{\vk}(\{x_0\})=0$.
\end{proof}

\section{Main theorem: special cases}
\subsection{More on the test functions $h_n$}

Recall
$$
 \fM_{n,\vk}=\cM_{\hus^{n-1}\vk}\cdots \cM_{\vk},\quad n\ge 1,
$$
so that 
$$
(\fM_{n,\vk}^*\Sigma_{\hat \us^n\vk})(dx)=\Sigma_{\vk}(dx).
$$
Let us define 
\[
\fP^{n+1}_{\vk}(z)=\fp_{\hat\us^n \vk}(T^{\circ n}(z))\fP^{n}_{\vk}(z),\quad \fP^{0}_{\vk}(z)=I.
\]
and
\begin{equation*}
	\Psi_\vk
	=\begin{cases}\begin{pmatrix}
			a_{\vk-1}\\
			a_{\vk} \end{pmatrix}, \quad \vk=2\us\vk,\\
		\\
		\begin{pmatrix}
			a_{\vk}\\
			a_{\vk+1} \end{pmatrix}, \quad \vk=1+2\us\vk.
	\end{cases}
	\quad
	\Upsilon_\vk^*
	=\begin{cases}
		\begin{pmatrix}
			1&0 \end{pmatrix}, \quad \vk=2\us\vk,\\
		\\
		\begin{pmatrix}
			0&1 \end{pmatrix}, \quad \vk=1+2\us\vk,
	\end{cases}
\end{equation*}
so that, see \eqref{eq:25feb-7},
$
\Phi_\vk=\Psi_\vk\Upsilon_\vk^*.
$
For two-dimensional vectors and matrices $A$ we define $A^{\star2}=AA^*$. 
\begin{lemma} Let $h_n(x)=(\fM_{n,\vk} h_0)(x)$ and $h_0(x)= w_x^0(T(0)) I$. There are constants $\{f^k_0\}_{k=1}^{n}$ (do not depend on $x$), $f_0^k=f_0^k(\vk)$, such that
\begin{align}\nonumber
h_n(x)=&
w^n_x(T(0))
\fP^n_{\vk}(T(0))  \fP^n_{\vk}(T(0))^*\\
+&\sum_{i=1}^{n}
 f^{i}_0(\vk)w^{n-i}_x(T(0))
(\fP^{n-i}_{\hat\us^{i}\vk}(T(0))
\Psi_{\hat\us^{i-1}\vk})^{\star2}
\label{eq:13jun-1c}
\end{align}
Moreover, let
\begin{equation}\label{eq:11sep-1c}
\fL_{n,i}(\vk)= w^{n-1-i}_0(T(0))
\langle
\fP^{n-1-i}_{\hat\us^{i}\vk}(T(0))\Psi_{\hat\us^{i-1}\vk},
\Upsilon_{\hat\us^{n-1}\vk}\rangle^2,\quad i<n.
\end{equation}
Then the following recurrence relation holds
\begin{equation}\label{eq:12sep-1}
f^n_0(\vk)=\sum_{i=1}^{n-1}\fL_{n,i}f^{i}_{0}(\vk)+w_0^{n-1}(T(0))\langle (\fP^{n-1}_{\vk} (T(0)))^{\star2}\Upsilon_{\hus^{n-1}\vk},\Upsilon_{\hus^{n-1}\vk}\rangle
\end{equation}

\end{lemma}

\begin{proof}
According to \eqref{eq:18mar-1},
we have
$$
h_1(x)=(\cM_\vk h_0)(x)=f^1_0w^0_x(T(0))\Psi_\vk\Psi_\vk^* +
w^1_x(T(0))\fp_\vk(T(0))^{\star2},
$$
with $f^1_0=w_0^0(T(0))\Upsilon^*_\vk \Upsilon_\vk=w_0^0(T(0))$. Assuming \eqref{eq:13jun-1c} and using \eqref{eq:18mar-1}, we get
\begin{align*}
\cM_{\hat\us^n \vk}h_n(x)&=w_x^{n+1}(T(0))\fP^{n+1}_{\vk}(T(0))^{\star2}\\
+&w_0^n(T(0))w_x^0(T(0))(\Phi_{\hat\us^n \vk}\fP^n_{\vk}(T(0)) )^{\star2}
\\
+&\sum_{i=1}^nf_0^iw_0^{n-i}(T(0))w_x^0(T(0))\left(\Phi_{\hat\us^n \vk}
 \fP^{n-i}_{\hat\us^{i}\vk}(T(0))
\Psi_{\hat\us^{i-1}\vk}\right)^{\star2}
 \\
+& \sum_{i=1}^nf_0^iw_x^{n-i+1}(T(0))
\left( \fp_{\hat\us^n \vk}(T^{\circ(n-i+1)}(0)) \fP^{n-i}_{\hat\us^{i}\vk}(T(0))
 \Psi_{\hat\us^{i-1}\vk}\right)^{\star2}
\end{align*}
By definition
\[
\fp_{\hat\us^n \vk}(T^{\circ(n-i+1)}(0))\fP^{n-i}_{\hat\us^{i}\vk}(T(0))=\fP^{n-i+1}_{\hat\us^{i}\vk}(T(0)).
\]
Thus, defining, see \eqref{eq:11sep-1c},
\begin{align*}
f_0^{n+1}=w_0^n(T(0))\Upsilon^*_{\hat\us^n \vk}(\fP^n_{\vk}(T(0)))^{\star2}\Upsilon_{\hat\us^n \vk}\\
+\sum_{i=1}^nf_0^iw_0^{n-i}(T(0))\Upsilon^*_{\hat\us^n \vk}(\fP^{n-i}_{\hat\us^{i}\vk}(T(0))
\Psi_{\hat\us^{i-1}\vk})^{\star2}
\Upsilon_{\hat\us^n \vk}\\
=\sum_{i=1}^{n}\fL_{n+1,i}f^{i}_{0}(\vk)+w_0^{n}(T(0))\langle (\fP^{n}_{\vk} (T(0)))^{\star2}\Upsilon_{\hus^{n}\vk},\Upsilon_{\hus^{n}\vk}\rangle
\end{align*}
\eqref{eq:13jun-1c} and \eqref{eq:12sep-1} follow by induction.
\end{proof}

\subsection{Equivalence Lemma}

\begin{lemma} \label{l:3dec} Let $\vk\in\cF_N$. 
The following statements are equivalent
\begin{itemize}
\item [(i)]Test functions $h_{n}(x)$ are uniformly bounded from below on $[-\xi,\xi]$, i.e., there exists $C_1>0$ such that $h_{n}(x)\ge C_1 I$, $x\in[-\xi,\xi]$.
\item[(ii)] The recurrence  coefficients $f^n_0(\vk)$ are  uniformly bounded from below, i.e.,
there exists $C_2>0$ such that $f^n_0(\vk)\ge C_2$.
\end{itemize}
\end{lemma}

\begin{proof} 
We will use the following relations for $a_\vk^2$
\begin{align*}
	a_{2\vk}^2+a_{2\vk+1}^2&=\lambda,\\
	a_{2\vk+1}^2a_{2\vk+2}^2&=a_{\vk+1}^2,
\end{align*}
see \cite{BellBeMou82,SY}. It yields an evident  upper bound for $a_\vk$: $a_\vk^2\le\lambda$ for all $\vk\in\bbZ_2$. It is known that $a_{\vk}\le 1$ for an arbitrary even $\vk\in\bbZ_2$ \cite{BellBeMou82}. It yields a lower bound for odd $\vk$: $a^2_{\vk}\ge \lambda-1>1$.
We will show that for $\vk\in\cF_N$  coefficients $a_{\vk}$ are uniformly bounded from below
\begin{equation}
\label{eq:28aug-1}
a_{\vk}^2\ge C_3=\frac{\lambda-1}{\lambda^N},\quad \lambda>2.
\end{equation}
Using the above estimates, the proof  follows by induction. Assume that $\vk'$ is odd. As the base of induction we have
$$
a^2_{2\vk'}
=\frac{a^2_{\vk'}}{a^2_{2\vk'-1}}
\ge \frac{\lambda-1}\lambda.
$$
Then we use  induction
\begin{equation}
\label{eq:28aug-2}
a^2_{2^{n+1}\vk'}=\frac{a^2_{2^n\vk'}}{a^2_{2^{n+1}\vk'-1}}\ge\frac{\lambda-1}{\lambda^{n}}\frac 1\lambda=\frac{\lambda-1}{\lambda^{n+1}}.
\end{equation}
Since an arbitrary element of $\cF_N$ contains at most $N$ consecutive zeros in its symbolic representation we have that $\vk=2^n\vk'$, were $0\le n\le N$ and $\vk'$ is odd.
Thus 
\eqref{eq:28aug-2} implies \eqref{eq:28aug-1}.

 (ii) $\Rightarrow$ (i). 
Since for $x\in[-\xi,\xi]$, all terms in the sum \eqref{eq:13jun-1c} are positive we can estimate $h_{n}(x)$ by two last summands
\begin{align*}
h_{n}(x) & \ge\nonumber
 f^{n}_0(\vk)w^{0}_x(T(0))
\Psi_{\hat\us^{n-1}\vk}
\Psi_{\hat\us^{n-1}\vk}^*\\
& \;\; + f^{n-1}_0(\vk)w^{1}_x(T(0))
\fp_{\hat\us^{n-1}\vk}(T(0))
\Psi_{\hat\us^{n-2}\vk}(
\fp_{\hat\us^{n-1}\vk}(T(0))
\Psi_{\hat\us^{n-2}\vk})^* \\
& \ge C_2\inf_{x\in[-\xi,\xi]}\{w^{0}_x(T(0)),w^{1}_x(T(0))\}\\
& \;\; \times\left(\Psi_{\hat\us^{n-1}\vk}
\Psi_{\hat\us^{n-1}\vk}^*
+
\fp_{\hat\us^{n-1}\vk}(T(0))
\Psi_{\hat\us^{n-2}\vk}(
\fp_{\hat\us^{n-1}\vk}(T(0))
\Psi_{\hat\us^{n-2}\vk})^*\right)
\end{align*}
Using the lower bound in \eqref{eq:27mar-1}, this implies
\begin{align}
h_{n}(x)
\ge \frac{C_2}{\lambda+\xi}
\left(\Psi_{\hat\us\vk'}
\Psi_{\hat\us\vk'}^*
+
\fp_{\hat\us\vk'}(T(0))
\Psi_{\vk'}(
\fp_{\hat\us\vk'}(T(0))
\Psi_{\vk'})^*\right),
\label{eq:28aug-3}
\end{align}
where $\vk'=\hus^{n-2}\vk$.

Now we note that for two dimensional vectors $a$ and $b$
$$
\det(a a^*+b b^*)={\det}^2\begin{pmatrix} a& b
\end{pmatrix},
$$
and therefore
$$
a a^*+b b^*\ge\frac{\det^2 \begin{pmatrix} a& b
\end{pmatrix}
}{\|a\|^2+\|b\|^2} I.
$$
For the vectors 
$$
a=\Psi_{\hat\us\vk'}\quad\text{and}\quad b=\fp_{\hat\us\vk'}(T(0))\Psi_{\vk'},
$$
by \eqref{eq:28aug-1}, we have
$$
{\det}^2\begin{pmatrix} a& b
\end{pmatrix}= a^2_{\hat\us\vk'+\e_1}a^2_{\vk'+\e_2}\ge C_3^2,\quad \e_{1},\e_{2}\in\{-1,0,1\}.
$$
For the norms of these vectors we have
$$
\|a\|^2=\|\Psi_{\hus\vk'}\|^2\le 2\lambda
$$
and
$$
\|b\|^2\leq\left(1+\frac{\|\Psi_{\hus\vk'}\|^2}{T(0)^2}\right)\|\Psi_{\vk'}\|^2\le 
\left(1+\frac{2}{\lambda}\right)2\lambda.
$$
Combining all these estimates with \eqref{eq:28aug-1} we get (i) with 
$$
C_1=\frac{C_2 C_3^2}{4(\lambda+\xi)(\lambda+1)}.
$$

 (i) $\Rightarrow$ (ii). By \eqref{eq:13jun-1c}, we have
 \begin{align}\nonumber
\langle h_{n}(0)\Upsilon_{\hus^{n}\vk}, \Upsilon_{\hus^{n}\vk}\rangle & =
w^n_0(T(0))
\langle  \fP^n_{\vk}(T(0))^*\Upsilon_{\hus^{n}\vk},\fP^n_{\vk}(T(0))^*\Upsilon_{\hus^{n}\vk}\rangle
\\
& \;\; +\sum_{i=1}^{n}
 f^{i}_0(\vk)w^{n-i}_0(T(0))
\langle\fP^{n-i}_{\hat\us^{i}\vk}(T(0))
\Psi_{\hat\us^{i-1}\vk},\Upsilon_{\hus^{n}\vk}\rangle^2
\label{eq:13aug-1c}.
\end{align}
Due to \eqref{eq:12sep-1}
\begin{align}\nonumber
f^{n+1}_0(\vk)=w_0^{n}(T(0))\langle (\fP^{n}_{\vk} (T(0)))^{\star 2}\Upsilon_{\hus^{n}\vk},\Upsilon_{\hus^{n}\vk}\rangle\\
\label{eq:13sep-1c}
+\sum_{i=1}^{n}\fL_{n+1,i}(\vk)f^{i}_{0}(\vk).
\end{align}
Since 
\begin{equation*}
\fL_{n+1,i}(\vk)= w^{n-i}_0(T(0))
\langle
\fP^{n-i}_{\hat\us^{i}\vk}(T(0))\Psi_{\hat\us^{i-1}\vk},
\Upsilon_{\hat\us^{n}\vk}\rangle^2
\end{equation*}
the right hand sides in 
 \eqref{eq:13aug-1c} and \eqref{eq:13sep-1c} coincide and we get the following interpolation property for  $f^n_0(\vk)$ 
\begin{equation}\label{eq:13sep-4}
 f^{n+1}_0(\vk)=\langle h_{n}(0)\Upsilon_{\hus^n\vk},\Upsilon_{\hus^n\vk}\rangle.
 \end{equation}
Consequently  $f^n_0(\vk)\ge C_1$ and the lemma is proved.
\end{proof}

\subsection{Uniform boundedness from below}

First let us mention {\em compactness} arguments. Let $K$ be a compact (symmetric ellipse)  containing the interval $[-\xi,\xi]$ which does not contain $-\lambda$. 
Then for all $n\ge n_0(K)$ and $z\in K$
$$
\frac 1 2 \le\left|\frac{z-T^{\circ n}(0)}{T^{\circ n}(0)}\right|\le 2.
$$
Due to the relation 
\[
w_z^n(T(0))=\frac{-T^{\circ(n+1)}(0)}{z-T^{\circ(n+1)}(0)}w_0^n(T(0)),
\] 
we can extend the uniform estimation of $w_z^n(T(0))$ for all $z\in K$ using \eqref{eq:27mar-1}. 
By \eqref{eq:Ben2}
$$
\|h_n(z)\|\le C(K)\sum_{k=0}^n \tr\, A^n_k,
$$
and by \eqref{eq:27mar-3} $\|h_n(z)\|\le C_1(K)$.

Now we mention a kind of \emph{weak convergence} for the sequence $\{f_0^n\}$ in the class $\cF_N$.  

\begin{lemma}\label{lem:43}
Let $\vk \in \cF_N$. Assume that $f^n_0\to 0$ along a certain subsequence of indices $\tilde I$. then for an arbitrary fixed $n_0$ and a sequence $k(n)\in \bbN$ with $k(n)\leq n_0$ we have
	$$\lim\limits_{\substack{  n\in \tilde I}}f^{n-k(n)}_0= 0.
	$$
\end{lemma}
\begin{proof}
	This follows from the estimate
	$$
	f^n_0\ge \cL_{n,n-1}f_0^{n-1}=w_0^0(T(0))\langle \Psi_{\hus^{n-2}\vk},\Upsilon_{\hus^{n-2}\vk}\rangle f_0^{n-1}.
	$$
	Therefore, by \eqref{eq:28aug-1}, $f^n_0\ge C f_0^{n-1}$ with a constant $C>0$ depending only on $N$.
\end{proof}

\begin{lemma}\label{l:21may-1}
Let $\vk\in\cF_N$. Then there exists $C>0$ such that $h_n(x)\ge C\cdot I$ for all $n$ and $x\in[-\xi,\xi]$.
\end{lemma}

\begin{proof}
Assuming the contrary, due to Lemma \ref{l:3dec} (specifically see \eqref{eq:13sep-4}),  we have a subsequence of indices $\tilde I$ such that
\begin{equation}\label{eq:3dec4-1}
\lim_{\substack{n\in \tilde I}}f^{n+1}_0=\lim_{n\in \tilde I} \langle h_n(0)\Upsilon_{\hus^n\vk}, \Upsilon_{\hus^n\vk}\rangle
=0.
\end{equation}
By Lemma \ref{lem:26} and Lemma \ref{lem:43} we can pass to another subsequence $ I$ such that $\hus^n\vk$ is odd, $\hus^{n-1}\vk$ is even and  \eqref{eq:3dec4-1} still holds for $n\in I$. 

Due to compactness of $\bbZ_2$ we have $\vk'$ and a subsequence $I_{\vk'}\subset I$ such that
\begin{equation}\label{eq:3dec4-2}
\lim_{n\in I_{\vk'}} \hus^n\vk=\vk'.
\end{equation}
Finally, due to compactness of $H^\infty(K)$ we choose a subsequence $I'_{\vk'}\subset I_{\vk'}$ such that for some $h_\infty \in H^\infty(K)$, uniformly on $K$,
\begin{equation}\label{eq:3dec4-3}
\lim_{ n\in I'_{\vk'}} h_n(z)=h_\infty(z).
\end{equation}

We will show that the matrix function $h_\infty(z)$ is of the form
\begin{equation}\label{eq:3dec4-4}
h_\infty(z)=\begin{pmatrix} h^\infty_{11}(z)&0\\0& 0 \end{pmatrix}.
\end{equation}
We note that the function $w_x^{0}(T(0)) = 1/(x+\lambda)$ is decreasing on $x \in [-\xi,\xi]$, but that for $k \ge 1$, the function $w_x^{k}(T(0))$ is increasing on $x \in [-\xi,\xi]$, because $T^{\circ j}(0) > 0$ for $j \ge 2$. For this reason, in \eqref{eq:Ben2} we separate the $k=0$ term and introduce
$$
h^\sharp_n(x)=\sum_{k=1}^n w^k_x(T(0)) A^n_k.
$$
By Lemma \ref{lem:43}, \eqref{eq:3dec4-1} implies $f^n_0\to 0$ for $n\in I'_{\vk'}, n\to\infty$. Since
by \eqref{eq:13jun-1c}, 
$$A^n_0=
f^{n}_0(\vk)
\Psi_{\hat\us^{n-1}\vk}\Psi_{\hat\us^{n-1}\vk}^*
$$
 we get by \eqref{eq:3dec4-3}
\begin{equation*}\label{}
\lim_{n\in I'_{\vk'} } h^\sharp_n(z)=h_\infty(z).
\end{equation*}
Since $h^\sharp_n(x)$ are positive and increasing (in $x$) on $[-\xi,\xi]$ we get that $h_\infty(x)$ is non-negative and non-decreasing in this interval.
By our choice of $I$ all terms $\hus^n\vk$ are odd. Therefore
$$
\Upsilon_{\hus^n\vk}=\Upsilon_{\vk'}=\begin{pmatrix} 0\\ 1
\end{pmatrix}.
$$
By \eqref{eq:3dec4-1}, we have
$$
\langle h_\infty(0)\Upsilon_{\vk'}, \Upsilon_{\vk'}\rangle
=0.
$$
By non-negativity and non-decreasing property we get
\begin{equation}\label{eq:3dec4-5}
\langle h_\infty(x)\Upsilon_{\vk'}, \Upsilon_{\vk'}\rangle
=0, 
\end{equation}
for all  $x\in[-\xi,0]$.
By analyticity \eqref{eq:3dec4-4} holds for all $x\in[-\xi,\xi]$. Non-negativity of $h_\infty(x)$ implies that off-diagonal entries also vanish. Thus 
\eqref{eq:3dec4-4} is proved.

Now we use that  $\hus^{n-1}\vk$ is even. Note that in particular
$$
\lim_{n\in I'_{\vk'}} \hus^{n-1}\vk=2\vk' .
$$
By passing to a further subsequence $I''_{\vk'} \subset I'_{\vk'}$ so that
\[
\tilde h_\infty(x)=\lim_{n\in I''_{\vk'}}h_{n-1}(x),
\]
and repeating the above arguments, we use
\[
\lim_{n\in I''_{\vk'}} \langle h_{n-1}(0)\Upsilon_{\hus^{(n-1)}\vk}, \Upsilon_{\hus^{(n-1)}\vk}\rangle = \lim_{n\in I''_{\vk'}} f_0^n = 0
\]
and using $\Upsilon_{\hus^{(n-1)}\vk} = \Upsilon_{2\vk'} = \binom 10$, we conclude as before that
\[
\tilde h_\infty(x) =\begin{pmatrix} 0 &0\\0& \tilde h^\infty_{22}(x)\end{pmatrix}.
\]

On the other hand
$$
h_n(x)=(\cM_{\hus^{n-1} \vk} h_{n-1})(x).
$$
Passing to the limit we have
\begin{align*}
h_\infty(x)=(\cM_{2 \vk'} \tilde h_{\infty})(x)=\frac 1 2\sum_{T(y)=x}
\fp_{2\vk'}(y) \begin{pmatrix}0 \\ 1\end{pmatrix}
\tilde h^\infty_{22}(y)
\begin{pmatrix} 0 &1 \end{pmatrix}\fp^*_{2\vk'}(y)\\
=\frac 1 2\sum_{T(y)=x}
\begin{pmatrix}0 \\ 1\end{pmatrix}
\tilde h^\infty_{22}(y)
\begin{pmatrix} 0 &1 \end{pmatrix}.
\end{align*}
Thus $h^\infty_{11}(x)=0$, i.e., $h^\infty(x)=0$, which contradicts the uniform bound $\tr\, h_\infty(x)\ge C>0$, see \eqref{eq:27mar-4b}.
\end{proof}

\subsection{Proof of the Main Theorem}

\begin{theorem}\label{thm:Fsigma}
If $\vk\in\cF_N$, then the spectrum of $J_\vk$ is singular continuous.
\end{theorem}

\begin{proof}
We have
$$
\int g(T(x))\tr\left(h_n(x) \Sigma_{\hus^n\vk}(dx)
\right)=
\int \tr\left(g(T(x))h_n(x) (\cM^*_{\hus^n\vk}\Sigma_{\hus^{n+1}\vk})(dx)
\right)
$$
$$
=
\int \tr\left((\cM_{\hus^n\vk}g(T(x))h_n(x)) \Sigma_{\hus^{n+1}\vk})(dx)
\right)=
\int g(x)\tr\left(h_{n+1}(x) \Sigma_{\hus^{n+1}\vk})(dx)
\right)
$$
Let $x_0 \in \sE$. Using a sequence of approximate $\delta$-functions $g_j(x) = (1 - j \lvert x - T^{\circ(n+1)}(x_0) \rvert)_+$ centered at $T^{\circ(n+1)}(x_0)$ and taking $j\to\infty$, we get
$$
\tr\left(h_{n+1}(T^{\circ(n+1)}(x_0)) \Sigma_{\hus^{n+1}\vk}(\{T^{\circ(n+1)}(x_0)\})\right)=
$$
$$
\tr\left(h_{n}(T^{\circ n}(x_0)) \Sigma_{\hus^{n}\vk}(\{T^{\circ n}(x_0)\})\right)+
\tr\left(h_{n}(-T^{\circ n}(x_0)) \Sigma_{\hus^{n}\vk}(\{-T^{\circ n}(x_0)\})\right)
$$
On the other hand since
$$
c_-\cdot I\le h_n(x)\le c_+\cdot I,\quad x\in[-\xi,\xi],
$$
we have
$$
\tr\left(h_{n}(-y_0) \Sigma_{\hus^{n}\vk}(\{-y_0\})\right)\ge c_-\tr\Sigma_{\hus^{n}\vk}(\{-y_0\})=
$$
$$
 c_-\tr\Sigma_{\hus^{n}\vk}(\{y_0\})\ge\frac {c_-}{c_+}
 \tr\left(h_{n}(y_0) \Sigma_{\hus^{n}\vk}(\{y_0\})\right).
$$
(in this step we used that the trace of the spectral matrix measure is even, since all diagonal Jacobi coefficients are zero). In combination we get
$$
\tr\left(h_{n+1}(T^{\circ(n+1)}(x_0)) \Sigma_{\hus^{n+1}\vk}(\{T^{\circ(n+1)}(x_0)\})\right)\ge
$$
$$
(1+\delta) \tr\left(h_{n}(T^{\circ n}(x_0)) \Sigma_{\hus^{n}\vk}(\{T^{\circ n}(x_0)\})\right)
$$
with $\delta=c_-/c_+$. If $\tr\,\Sigma_\vk(\{x_0\})>0$, then $\tr\,\left(h_0(x_0)\Sigma_\vk(\{x_0\})\right)>0$ and the sequence
$ \tr\left(h_{n}(T^{\circ n}(x_0)) \Sigma_{\hus^{n}\vk}(\{T^{\circ n}(x_0)\})\right)$ is unbounded. On the other hand, the measures $\Sigma_{\hus^n\vk}$ and functions $h_n(x)$ are uniformly bounded.
\end{proof}

\begin{proof}[Proof of Theorem~\ref{mt}]
Since $\bbZ_2 = \cG_\delta \cup \cF_\sigma$ and $\cF_\sigma = \cup_N \cF_N$, this follows from Theorems~\ref{thm:Gdelta} and \ref{thm:Fsigma}.
\end{proof}


\end{document}